\documentclass[11pt,a4paper]{amsart}
\usepackage{amsfonts,amsthm,amsmath,amssymb,latexsym}
\usepackage{graphicx}
\usepackage[T1]{fontenc}
\usepackage{mathpazo}

\usepackage{hyperref}
\hypersetup{colorlinks=true, urlcolor=blue, citecolor=blue, linkcolor=blue}
\usepackage{url}
\newtheorem{thm}{Theorem}
\newtheorem{cor}[thm]{Corollary}
\newtheorem{lem}[thm]{Lemma}

\theoremstyle{definition}

\theoremstyle{remark}

\newcommand{\dd}{\mathinner{\ldotp\ldotp}}
\title{Palindromes in starlike trees}%

\begin{document}


\author{Amy Glen}

\address{Amy Glen \newline
\indent School of Engineering \& Information Technology \newline
\indent Murdoch University \newline
\indent  90 South Street \newline
\indent Murdoch, WA 6150 AUSTRALIA}%
\email{\href{mailto:A.Glen@murdoch.edu.au}{A.Glen@murdoch.edu.au}}

\author{Jamie Simpson}

\address{Jamie Simpson \newline
\indent Department of Mathematics and Statistics \newline
\indent Curtin University \newline
\indent Bentley, WA 6102 AUSTRALIA}%
\email{\href{mailto:Jamie.Simpson@curtin.edu.au}{Jamie.Simpson@curtin.edu.au}}

\author{W. F. Smyth}
\address{W. F. Smyth \newline
\indent Department of Computing and Software \newline
\indent McMaster University \newline
\indent Hamilton, Ontario L8S4K1 CANADA}%
\email{\href{mailto:smyth@mcmaster.ca}{smyth@mcmaster.ca}}

\begin{abstract} In this note, we obtain an upper bound on the maximum number of distinct non-empty palindromes in starlike trees. This bound implies, in particular, that there are at most $4n$ distinct non-empty palindromes in a starlike tree with three branches each of length~$n$ --- for such starlike trees labelled with a binary alphabet, we sharpen the upper bound to $4n-1$ and conjecture that the actual maximum is $4n-2$. It is intriguing that this simple conjecture seems difficult to prove, in contrast to the straightforward proof of the bound.
\end{abstract}

\maketitle

\section{Introduction}

We use the usual notation and terminology from graph theory and combinatorics on words.  

A \emph{word} of $n$ elements is represented by an array $x=x[1 \dd n]$, with $x[i]$ being
the $i$th element and $x[i\dd j]$ the \emph{factor} of elements from
position $i$ to position $j$.  If $i=1$ then the factor is a
\emph{prefix} and if $j=n$ it is a \emph{suffix}. The letters in $x$
come from some \emph{alphabet} $A$.  The \emph{length} of~$x$,
written $|x|$, is the number of letters $x$ contains.
 If $x=x[1\dd n]$ then the
\emph{reverse} of $x$, written $R(x)$, is $x[n]x[n-1]\cdots x[1].$ A 
word $x$ that satisfies $x = R(x)$ is called a \emph{palindrome}.

A \emph{starlike tree} $T$ is a tree consisting of a root vertex, called the \textit{central vertex}, from which there extends 3 or more \textit{branches} (i.e., simple paths) where each edge of a path directed from the central vertex to the terminal vertex (leaf) of a branch is labelled 
with a single letter of an alphabet $A$. 
Thus every path from the central vertex to a leaf in the tree,
as well as every simple path passing in reverse order from a leaf through the central vertex to the leaf of another branch,
constitutes a word. 
If a starlike tree $T$ consists of $k$ branches, each of length $n$, we say that $T$ is a $(k,n)-$\emph{starlike tree}.

\begin{figure}[htb!]
\includegraphics[scale=0.45]{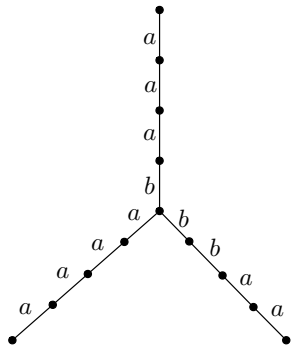}
\caption{\footnotesize A $(3,4)-$starlike tree with edges labelled using the binary alphabet $\{a,b\}$.}  \label{tree}
\end{figure}

The maximum number of distinct non-empty palindromes in a word of length $n$ is $n$ (see Theorem~\ref{DJP} below).   Moving from  words to graphs we can think of this  as the maximum number of distinct non-empty palindromes in an edge-labelled path $P_n$ of length $n$ where the labels are single letters.  This suggests extending the problem to other graphs.  In \cite{JS} it was shown that the maximum number of distinct non-empty palindromes in a cycle $C_n$ is less than $5n/3$.  For $n$ divisible by 3 the so-called \textit{Biggles Words} contain $5n/3-2$ distinct palindromes, so the bound is almost sharp. Brlek, Lafreni\'{e}re and Proven\c{c}al \cite{BLP} studied the palindromic complexity of trees and constructed families of trees with $n$ edges containing $\Theta(n^{1.5})$ distinct palindromes. They conjectured that there are no trees with asymptotically  larger palindromic complexity than that, and this was later proved by  Gawrychowski, Kociumaka,  Rytter and Wale\'{n} \cite{GKRW}.

In this note, we consider the maximum number of distinct non-empty palindromes
that can exist in a $(k,n)-$starlike tree.
We call this number $P(k,n)$ and prove that $P(k,n) \le (1+ \binom{k}{2})n$.
For $k=3$ this gives $P(3,n) \le 4n$, but
for trees labelled with a binary alphabet we sharpen this result to $P(3,n) \le 4n-1$. 
On the basis of computational evidence,
we conjecture that for $k = 3$ the best possible bound is $4n-2$.
Trees attaining this bound are easily found, but mysteriously, it seems very difficult to prove.

\section{Results}

The following well-known result is due to Droubay, Justin, and
Pirillo~\cite{DJP}. We give a proof since its  ideas will be
used later.

\begin{lem}\label{DJP} The number of distinct non-empty palindromes in a word
of length $n$ is at most $n$.
\end{lem}
\begin{proof} If two palindromes end at the same place then the
shorter is a suffix of the longer.  It is therefore also a prefix of
the longer and so has occurred earlier in the word. Thus at each
position there is the end of at most one palindrome making its first
appearance in the word. The lemma follows.
\end{proof}

Of course it is also true that each position in a word
can be the starting point of the last occurrence of at most one palindrome.
Note that a position $i$ that marks the end of the first appearance of a palindrome
in $x[1..n]$ also marks the start of the last occurrence $j = n-i+1$ of the same palindrome
in $R(x)$.

\begin{thm} \label{T:main} An edge-labelled starlike tree with branches 
$b_1, b_2,\dots, b_k$, where $|b_1| \ge |b_2| \ge \cdots \ge |b_k|$,
contains at most $$|b_1|+\sum_{i=2}^k (i-1)|b_i|$$ distinct non-empty palindromes.
\end{thm}
\begin{proof}
We say that a palindrome within a branch $b_i$ is \emph{local},
while one that overlaps a path $R(b_i)b_j$, $j > i$, through the central vertex is \emph{overlapping}.

The path $R(b_1)b_2$ contains $|b_1|+|b_2|$ edge labels and therefore,
by Lemma~\ref{DJP}, contains at most $|b_1|+|b_2|$ distinct palindromes.
Now consider the path $R(b_1)b_3$.
This contains at most $|b_1|+|b_3|$ distinct palindromes,
with at most $|b_3|$ of their first appearances ending in $b_3$.
Palindromes local to $b_1$  would have been counted in the path $R(b_1)b_2$.
Thus, in addition to these palindromes, there are at most $|b_3|$ other
palindromes in $R(b_2)b_3$, whether local or overlapping.
Similarly, there are at most $|b_i|$ new palindromes
in each path $R(b_1)b_i$ for $i = 4,5,\ldots,k$.
Thus the total number of new palindromes in paths $R(b_1)b_i$, $i = 2,3,\ldots,k$,
is at most $\sum_{i=1}^k |b_i|$.

Now consider the paths $R(b_2)b_i$ for $3 \le i \le k$.
The set of palindromes in $R(b_2)$ is of course exactly the set
of palindromes in $b_2$.
These local palindromes fall into two types, as follows. 
\begin{itemize}
\item[Type 1:]
Those palindromes counted in $b_{12} = R(b_1)b_2$ because their first occurrences were in $b_{12}$,
thus not in $R(b_1)$. These will of course also occur in $R(b_2)b_i$,
but will not be counted a second time.
\item[Type 2:]
Those palindromes \emph{not} counted in $b_{12}$.
These palindromes must therefore not have their first occurrences in $R(b_1)b_2$,
and so must have occurred (and been counted) previously. 
These palindromes will therefore not be counted a second time.
\end{itemize}
Thus in $R(b_2)b_i$ there will be no new palindromes local to $b_2$,
only (local or overlapping) palindromes ending in $b_i$, of which there will be at most $\sum_{i=3}^k |b_i|$ altogether, by Lemma~\ref{DJP}.

Considering now all the paths $R(b_i)b_j$, $i = 1,2,\ldots,k-1,\ j = i+1,i+2,\ldots,k$,
we see that the maximum number of palindromes in all paths of the starlike tree is
\begin{eqnarray*}
&& \sum_{i=1}^k |b_i| + \sum_{i=3}^k |b_i| + \dots +\sum_{i=k-1}^k |b_i|+|b_k|\\
&=& |b_1| + \sum_{i=2}^k (i-1)|b_i|,
\end{eqnarray*}
as required.
\end{proof}

\begin{cor} \label{C1}
$$P(k,n) \le \left(1+ \binom{k}{2}\right)n.$$
\end{cor}
\begin{proof} Substitute $n$ for each $|b_i|$ in the theorem.
\end{proof}

Table~\ref{table} (below) shows values of $P(k,n)$ for low values of $k$ and $n$ when we are restricted to a \textit{binary alphabet}. We see that the upper bounds are far from sharp.  
{
\begin{table}[htb!]
\[\begin{array}{lccc}
n & k=3 & k=4 & k=5\\
   1  & 3,4 & 4,7 & 4,11\\
   2  & 6,8 & 8,14 & 9,22\\
   3  & 10,12 & 14,21 & 16,33\\
   4  & 14,16 & 20,28 & 24,44\\
   5  & 18,20 & 26,35 & 32,55\\
 \end{array} \]
 
\caption{\footnotesize Maximum number of palindromes in various starlike trees of fixed branch length.  The first value in each cell is $P(k,n)$ and the second is the upper bound given by  Corollary~\ref{C1}.} \label{table}
\end{table}
}

\newpage

Using larger alphabets does not seem to increase the maxima except in the case of a starlike tree with five branches of length $1$.  With a binary alphabet we get at most four palindromes with branches labelled $a, a, a,b,b$; with a ternary alphabet we get five palindromes using $a,a,b,b,c$.

One might expect there to be an easy induction proof, but there is not. In the case of starlike trees with $3$ branches of length $n$ labelled by a binary alphabet, Corollary~\ref{C1} gives an upper bound of $4n$. One might assume that by adding an extra letter to each branch you could only get at most $4$ new palindromes, but this is not implied by our proof. If the three branches are labelled $A$, $B$ and $C$, adding an extra letter to the $A$ branch can give an extra palindrome in the $AB$ branch and in the $AC$ branch, so $2$ more palindromes starting in the $A$ branch. Also $2$ more starting in each of the branches $B$ and $C$, so up to $6$ new palindromes (not 4) altogether.

We conjecture that, for all $n \geq 2$, $P(3,n) = 4n-2$. This bound can be attained using branches labelled $a^n$, $ba^{n-1}$ and $bba^{n-2}$ (see Figure~\ref{tree}), but it seems very difficult to prove. One can make similar conjectures for larger $k$  but in these cases there are many examples attaining the maxima, none of which look suitably canonical.

The following is a slight improvement on the bound for $P(3,n)$ when the alphabet is binary.

\begin{thm} The maximum number of distinct non-empty palindromes arising when a binary alphabet is used to edge-label a starlike tree with three length $n$ branches is at most $4n-1$.
\end{thm}
\begin{proof}
Label the three branches $x$, $y$ and $z$.
Since our alphabet is binary, at least two of these have the same $n$th letter,
say $\alpha$.
Without loss of generality, suppose $y[n] = z[n] = \alpha$.
There are at most $2n$ palindromes in $R(x)y$ and at most $n$ more distinct palindromes in $R(y)z$.
We claim however that there are at most $n-1$ new ones in $R(x)z$. 
For suppose that a palindrome in $R(x)z$ ends at $z[n]$.
Then it must also occur in $R(x)y$ ending at $y[n]$, and so it has already been counted.
Thus there are at most $4n-1$ distinct palindromes in the starlike tree.
\end{proof}
Similar results hold whenever paths in the tree share a common suffix.

\bibliographystyle{amsplain}

\end{document}